\documentclass[11pt]{amsart}
\usepackage[utf8]{inputenc}

\usepackage{amssymb}
\usepackage{graphics}
\usepackage{graphicx}
\usepackage{amscd}
\usepackage{color}
\usepackage{amsmath, amsthm, epsfig,  psfrag}
\usepackage{pinlabel}

\newtheorem{theorem}{Theorem}[section]
\newtheorem{lemma}[theorem]{Lemma}
\newtheorem{prop}[theorem]{Proposition}
\newtheorem{cor}[theorem]{Corollary}

\theoremstyle{definition}
\newtheorem{definition}[theorem]{Definition}
\newtheorem{remark}[theorem]{Remark}

\newtheorem{question}[theorem]{Question}

\newcommand{\abs}[1]{{\bigl\lvert #1\bigr\rvert}}
\newcommand{\connsum}{{\mathbin{\#}}}
\renewcommand{\d}{\partial}
\newcommand{\lcan}{{\lambda_{\mathrm{can}}}}
\newcommand{\lie}[1]{{\mathcal{L}_{#1}}}
\newcommand{\Exo}{{\mathrm{exo}}}
\newcommand{\p}{{\mathbf p}}
\newcommand{\PS}{{\mathrm{PS}}}
\newcommand{\q}{{\mathbf q}}
\newcommand{\Stab}{{\mathrm{stab}}}
\newcommand{\Std}{{\mathrm{std}}}

\newcommand{\D}{\mathcal{D}}
\newcommand{\PP}{\mathcal{P}}

\newcommand{\CC}{\mathbb{C}}

\newcommand{\RR}{\mathbb{R}}
\newcommand{\ZZ}{\mathbb{Z}}

\newcommand{\restricted}[2]{{\left.{#1}\right|_{#2}}}

\DeclareMathOperator{\GL}{GL}

\DeclareMathOperator{\tb}{tb}
\DeclareMathOperator{\U}{U}
\DeclareMathOperator{\Span}{Span}

\newcommand{\defin}[1]{\textbf{#1}}

\begin{document}

\author[E.\ Murphy]{Emmy Murphy}
\address[E. Murphy]{Massachusetts Institute of Technology
  Department of Mathematics \\
  77 Massachusetts Avenue \\
  Cambridge, MA 02139, USA}
\email{e\_murphy@mit.edu}

\author[K.\ Niederkrüger]{Klaus Niederkrüger}
\address[K.\ Niederkrüger]{
  Institut de mathématiques de Toulouse\\
  Université Paul Sabatier -- Toulouse III\\
  118 route de Narbonne\\
  F-31062 Toulouse Cedex 9\\
  FRANCE}
\email{niederkr@math.univ-toulouse.fr}

\author[O.\ Plamenevskaya]{Olga Plamenevskaya}
\address[O.\ Plamenevskaya]{Department of Mathematics, SUNY Stony
  Brook, Stony Brook, NY 11794, USA}
\email{olga@math.sunysb.edu}

\author[A.\ I.\ Stipsicz]{András I. Stipsicz}
\address[A.\ I.\ Stipsicz]{Rényi Institute of Mathematics, Budapest,
  Reáltanoda utca 13-15., HUNGARY}
\email{stipsicz.andras@renyi.mta.hu}

\title{Loose Legendrians and the plastikstufe}

\begin{abstract}
  We show that the presence of a plastikstufe induces a certain degree
  of flexibility in contact manifolds of dimension $2n+1>3$.
  More precisely, we prove that every Legendrian knot whose complement
  contains a ``nice'' plastikstufe can be destabilized (and, as a
  consequence, is loose).
  As an application, it follows in certain situations that two
  non-isomorphic contact structures become isomorphic after
  connect-summing with a manifold containing a plastikstufe.
\end{abstract}

\maketitle


\section{Introduction} 
\label{sec:intro}

In dimension $3$, it has been known for a long time that contact
structures containing a topological object called an
\emph{overtwisted} disk are ``flexible'', in the sense that two
overtwisted contact structures which are homotopic as oriented
$2$-plane fields are also isotopic \cite{EliashbergClassification}.
Often this property is phrased as saying that overtwisted contact
structures on $3$-manifolds satisfy an $h$-principle.
A second important property of all contact structures containing
overtwisted disks is that such contact $3$-manifolds are not
symplectically fillable.
In higher dimensions, the quest for the right definition of
``overtwisted'' contact structures has been going on for a while, but
we are far from having a definitive answer.
On the one hand, there is a variety of special submanifolds,
plastikstufes and bLobs \cite{Nie, MNW}, whose presence in a contact
manifold is known to obstruct fillability.
On the other hand, certain related objects possess flexibility
properties desired for "overtwisted" contact structures.
More specifically, a class of Legendrian knots, called ``loose''
Legendrians, satisfies a version of $h$-principle \cite{Mur}, and
Weinstein cobordisms obtained by attaching symplectic handles along
loose knots are governed by a ``symplectic $h$-cobordism theorem''
\cite{CE}.
(Roughly speaking, loose Legendrian knots in higher dimensional
contact manifolds are those that contain a sufficiently wide ``kink''
whose projection to a $3$-dimensional subspace is a stabilized
Legendrian arc.
See Section~\ref{sec:loose} for a precise definition.)
In this paper, we hope to shed some light on flexibility in high
dimensions by studying loose Legendrians in contact manifolds
containing a plastikstufe.
(The latter is a foliated submanifold of maximal dimension that is a
product of an overtwisted disk and a closed manifold.)
Our main result is the following 

\begin{theorem}\label{theorem: p-loose}
  Let $(M^{2n+1}, \xi)$ be any contact manifold containing a small
  plastikstufe~$\PP_B$ with spherical binding and trivial rotation.
  Then any Legendrian in $(M, \xi)$ which is disjoint from $\PP_B$ is
  loose.
\end{theorem}

\noindent (The definitions of the various terms in the above theorem
are deferred to Sections~\ref{sec:plastikstufe} and \ref{sec:loose}.)
Examples of $PS$-overtwisted contact manifolds satisfying the
hypotheses of Theorem~\ref{theorem: p-loose} have been constructed in
\cite{EP}.
The proof of Theorem~\ref{theorem: p-loose} is based on the proof of
the corresponding statement concerning overtwisted contact
$3$-manifolds, together with a certain isotopy gained from a classical
$h$-principle \cite{PDR} to bring a part of the given Legendrian with
respect to the plastikstufe into product position.
As a consequence of the above theorem, we establish a flexibility result motivated by a familiar
$3$-dimensional fact: if any two contact structures on a $3$-manifold
are homotopic as plane fields, they become isotopic as contact
structures after connect-summing with any overtwisted contact
manifold.
The following theorem is a corollary of Theorem~\ref{theorem: p-loose}
and of the results of Cieliebak-Eliashberg on flexible Weinstein
cobordisms \cite{CE}.
Different versions of the symplectic $h$-cobordism theorem
\cite{CE} imply several versions of our result below.
\begin{theorem}\label{theorem: exo=std}
  Consider two contact structures $\xi_0, \xi_1$ on a manifold $Y$ of
  dimension $2n+1>3$.
  Let $(M, \xi_\PS)$ be a simply connected manifold containing a small
  plastikstufe with spherical binding and trivial rotation.
  Assume that one of the following conditions holds:
  \begin{enumerate}
  \item There exists a manifold~$W$ of dimension $2n+2$ with $\d W=Y$,
    and $W$ carries two Stein structures~$J_0$, $J_1$ such that $(W,
    J_0)$ is a filling for $\xi_0$, $(W, J_1)$ is a filling for
    $\xi_1$, and $J_0$, $J_1$ are homotopic through almost complex
    structures, or
  \item There exists a Stein cobordism $(W, J)$
    from $(Y, \xi_0)$ to $(Y, \xi_1)$ such that $W$ is smoothly a
    product cobordism $Y\times [0,1]$.
  \end{enumerate}
  Then $(M, \xi_\PS) \connsum (Y, \xi_0)$ is contactomorphic to $(M,
  \xi_\PS) \connsum (Y, \xi_1)$.
\end{theorem}

According to  \cite{CE}, the requirement about the existence of Stein structures on $W$ can be exchanged to Weinstein structures. (For the definition of a 
Weinstein structure, see Section~\ref{sec:loose}.) Indeed, a Stein structure induces a 
Weinstein structure, and any Weinstein structure (after possibly a  homotopy through
Weinstein structures) can be induced by a Stein structure. In addition,
two Stein structures are homotopic (through Stein structures) if and only if 
the induced Weinstein structures are homotopic (through Weinstein structures).
In conclusion, in the above theorem the two notions are interchangeable.
In fact, in accordance with the results we quote from \cite{CE}, the proofs in Section~\ref{sec:loose}
will be phrased in the language of Weinstein manifolds, and we appeal to the
above interchangeability principle.
Unlike the $3$-dimensional situation, we are unable to prove any
flexibility result in the absence of Stein/Weinstein fillings or cobordisms.
%

Mark McLean \cite{Mc} constructed an infinite family of examples of
exotic Stein structures of finite type on $\RR^{2n+2}$ for
$n>1$. (Further such examples are given in \cite{AS}.)
A sphere $S^{2n+1}$ that encloses all the critical points of the
plurisubharmonic Morse function in any of the above exotic Stein
$\RR^{2n+2}$ carries a contact structure $\xi_\Exo$ filled by a
non-standard Stein ball.
It is known \cite{Mc, AS} that $(S^{2n+1}, \xi_\Exo)$ is not
contactomorphic to $(S^{2n+1}, \xi_\Std)$. (Actually, there are
infinitely many distinct contact spheres among these exotic ones.)
Now, let $(S^{2n+1}, \xi_\PS)$ be a sphere containing a plastikstufe
as in Theorem~\ref{theorem: exo=std} (the existence of which, with the
right rotation class in every dimension, will be checked in
Section~\ref{sec:hprinciples}, cf. Proposition~\ref{prop:
  PSwithZeroTwist}).
\begin{cor}
  With the notations as above, $(S^{2n+1}, \xi_\Exo \connsum \xi_\PS)$
  is contactomorphic to $(S^{2n+1}, \xi_\Std \connsum \xi_\PS) \cong
  (S^{2n+1}, \xi_\PS)$. 
\end{cor}
\begin{proof}
  This is an immediate consequence of the previous theorem: one could
  use either condition~(1), applied to fillings by the standard and
  non-standard Stein ball, or condition~(2), applied to the cobordism
  obtained by puncturing the non-standard Stein ball.
\end{proof}

%
The paper is organized as follows.
In Section~\ref{sec:otdisk}, we quickly review the $3$-dimensional
case: if $M^3$ is overtwisted, every Legendrian knot~$L \subset M$
disjoint from an overtwisted disk can be destabilized.
The destabilization is realized by taking a connected sum with the
boundary of an overtwisted disk.
(While this seems to be a ``folklore'' statement, explicit proofs are
absent from the literature.)
The definition and properties of plastikstufes are reviewed in
Section~\ref{sec:plastikstufe}.
In Sections~\ref{sec:hprinciples} and \ref{sec:loose} we 
recall the definition of loose knots and finally 
prove Theorem~\ref{theorem: exo=std}.

\subsection*{Acknowledgments}

The main results of the present work were found when the authors
visited the American Institute of Mathematics (AIM), as participants
of the \emph{"Contact topology in higher dimensions"} workshop.
We would like to thank AIM for their hospitality, the ESF and the NSF
for the support, and other workshop participants for stimulating conversations.
We are grateful to John Etnyre for useful conversations and e-mail
correspondence. The third author would also like to thank Yasha Eliashberg, Ko Honda and Thomas Vogel for answering some questions.
KN was partially supported by Grant \emph{ANR-10-JCJC 0102} of the
\emph{Agence Nationale de la Recherche}. 
OP was partially supported by NSF grant DMS-1105674.
AS was partially supported by ERC Grant LDTBud and by \emph{ADT
  Lend\"ulet}.

\section{The $3$-dimensional case: sliding over an overtwisted
  disk}\label{sec:otdisk}

In this section  we prove that a Legendrian knot  in the complement of
an overtwisted disk can be destabilized.
More precisely, we need a local statement: assuming that the knot is
close to the overtwisted disk, we would like to find a destabilization
by modifying the knot \emph{only} in the neighborhood of the
overtwisted disk.
For this local modification, we restrict to the case of a ``simple''
overtwisted disk, that is, an embedded disk whose characteristic
foliation is isomorphic to the one shown in Figure~\ref{fig: ot-fol}:
the boundary of the disk is the closed leaf of the foliation, and
there is a unique (elliptic) singularity inside the disk.
(Recall that an overtwisted disk is any embedded disk with
Legendrian boundary with $\tb=0$. Standard arguments resting on 
Giroux Flexibility imply that a manifold containing an overtwisted disk also 
contains a simple one.)
Although it is a known fact that any Legendrian knot in the complement
of an overtwisted disk can be destabilized (cf.\ \cite{Dymara}), we
were unable to find the above "local" version in the literature, and
therefore (for the sake of completeness) we include the arguments
below.

\begin{figure}[ht]
  \includegraphics[scale=1.3, keepaspectratio]{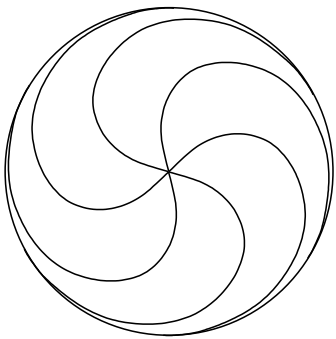}
  \caption{A simple overtwisted disk.}
  \label{fig: ot-fol}
\end{figure} 

Before proceeding any further, we recall the definition of the
stabilization for Legendrian knots in $\RR^3$.
(For a general discussion on Legendrian knots, see,
e.g. \cite{EtnyreKnots}.)
Suppose that in a Darboux chart containing a strand of the Legendrian
knot~$L$, the front projection of this strand has a cusp.
To stabilize $L$, remove this cusp in the projection and replace it by
a kink, as shown in Figure~\ref{fig: stab-strand}.
It is not hard to check that, up to Legendrian isotopy, the stabilized
knot~$L_\Stab$ is independent of the choice of the Darboux coordinates
and of the point where the stabilization is performed.
A more invariant definition can be given by using convex surfaces:
$L_\Stab$ is a stabilization of $L$ if $L \cup L_\Stab$ is the
boundary of a convex annulus whose dividing set consists of a single
arc with endpoints on $L_\Stab$.

\begin{figure}[ht]
  \includegraphics[scale=1.1, keepaspectratio]{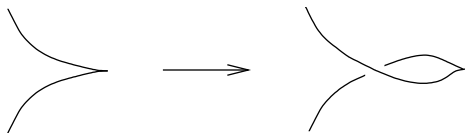}
  \caption{Stabilization of a Legendrian knot in dimension~$3$.}
  \label{fig: stab-strand}
\end{figure}

\begin{theorem}\label{theorem: ot-disk destabilizes}
  Let $(M^3, \xi)$ be an overtwisted contact manifold.
  Suppose that $L$ is a Legendrian knot in the complement of a
  simple overtwisted disk $\D_{ot}$.
  Then $L$ can be destabilized, that is, it is the stabilization of
  another Legendrian knot.
  A destabilization of $L$ is given by the Legendrian connected sum $L
  \connsum \d \D_{ot}$ of $L$ and the boundary of the overtwisted
  disk.
\end{theorem}

\begin{cor}
  If $(M^3, \xi)$ is an overtwisted contact manifold, then any
  Legendrian knot in the complement of any overtwisted disk can be
  destabilized. \qed
\end{cor}

The above corollary follows immediately from Theorem~\ref{theorem:
  ot-disk destabilizes}:
as we already remarked, any overtwisted contact 
structure on the knot complement contains a simple
overtwisted disk.
(Alternatively, the corollary can be derived from results of \cite{EF} or
\cite{Et}.)
It is useful, however, to have an explicit destabilization procedure
given by Theorem~\ref{theorem: ot-disk destabilizes}.
Our proof of Theorem~\ref{theorem: ot-disk destabilizes}, or at least
its main idea, is essentially borrowed from
\cite[Proposition~3.22]{Vog}, although the statement of \cite{Vog} is
different from ours.
(We are also indebted to John Etnyre, who pointed out Vogel's
lemma to us.)
We modify the proof from \cite{Vog} to adapt it to our present
purposes.

\begin{proof}[Proof of Theorem~\ref{theorem: ot-disk destabilizes}] 
  We will use convex surface theory \cite{Gi, Ho} to see the
  destabilization.
  First, notice that the overtwisted disk from Figure~\ref{fig:
    ot-fol} is convex, and its dividing set consists of a unique
  closed curve.
  By Giroux's Flexibility principle, we can isotop the disk (while
  keeping its Legendrian boundary fixed) to obtain a convex disk with
  any characteristic foliation divided by the same curve.
  Consider the foliation shown in Figure~\ref{fig: ot-fol2}.
  It is clear that this foliation can be carried by a convex surface
  and is divided by the same closed curve.
  Thus, after an isotopy, we can assume that the foliation on
  $\D_{ot}$ looks like Figure~\ref{fig: ot-fol2} and has the following
  features:
  \begin{itemize}
  \item [(i)] a family of straight Legendrian arcs (vertical lines on
    Figure~\ref{fig: ot-fol2}) separate $\D_{ot}$ into two
    half-disks;
  \item [(ii)] there are smooth Legendrian closed curves that run
    close to the boundary of these half-disks and are formed by the
    union of the leaves of the foliation.
  \end{itemize}
  Now, consider a (piecewise smooth) Legendrian curve formed by the
  left semicircle of $\d D_{ot}$ and the leftmost vertical Legendrian
  arc.
  Let $D_l$ denote the half-disk bounded by this curve in $\D_{ot}$.
  A disk~$D_r$ is defined similarly on the right side of $\D_{ot}$.

  \begin{figure}[ht]
    \labellist
    \small\hair 2pt
    \endlabellist
    \centering	
    \includegraphics[scale=1.1, keepaspectratio]{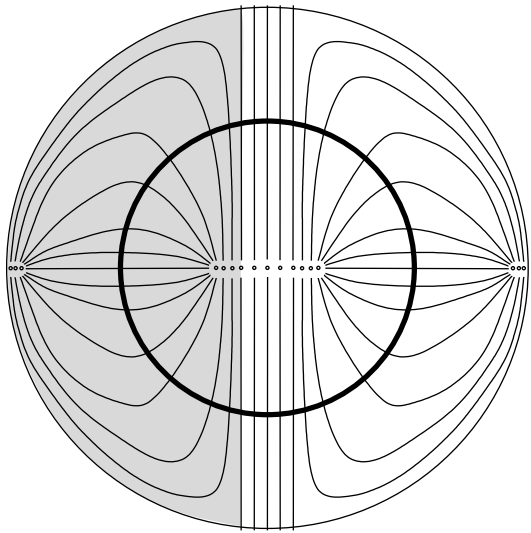}
    \caption{The overtwisted disk~$\D_{ot}$ can be endowed with
      another characteristic foliation divided by the same dividing
      set (the thick circle).
      The foliation in the figure has a collection of singular
      points at the top and bottom of the boundary circle and on
      the horizontal diameter (near the endpoints and the
      center).
      A family of vertical leaves separates the right and the left
      half-disks.
      The vertical leaves go from singular points at the top
      (resp.\ bottom) of the boundary circle to singular points
      near the center.
      The right and left halves of the disk are foliated by arcs
      connecting singularities near the endpoints of the
      horizontal diameter and those near the center of
      $\D_{ot}$.
      The arcs located near $\d \D_{ot}$ have vertical tangencies,
      thus the union of an arc on the top and the corresponding
      arc on the bottom (together with the singular points they
      connect) is a smooth Legendrian knot.
      The half-disk~$D_l$ on the left side of $\D_{ot}$ is cut off by
      the leftmost vertical Legendrian arc.}
    \label{fig: ot-fol2}
  \end{figure}

  Recall that the Thurston-Bennequin number of a Legendrian knot~$K$
  can be computed \cite{Ho} from the dividing set
  $\Gamma_{\Sigma}$ on its convex Seifert surface~$\Sigma$:
  \begin{equation}\label{eq: tb-rot-G}
    \tb(K) = -\frac{1}{2}\, \abs{\Gamma \cap K} \;.
  \end{equation}
  From this formula, $\tb (\d D_l)= \tb(\d D_r)=-1$.
  Note that the Thurston-Bennequin number is defined even for
  piecewise smooth curves: indeed, linking with the transverse
  push-off is still well-defined.
  The formula~\eqref{eq: tb-rot-G} also holds in this context, which
  can be seen, for example, via approximation by smooth Legendrian
  curves.
  We will be attaching the overtwisted disk in two steps, first one
  half, that is, $D_l$, then the other, i.e. the Legendrian arcs and
  $D_r$.
  The idea is that half of $\D_{ot}$ is easier to control: indeed,
  half of an overtwisted disk represents a bypass, and bypasses can be
  found in tight contact manifolds \cite{Ho}.
  Now we are ready to prove the theorem.
  Since the statement is local, it suffices to establish the claim for
  a small Legendrian unknot~$L_0$ with $\tb(L_0)=-1$, located near the
  overtwisted disk $\D_{ot}$ (and disjoint from it).
  Let $L$ be another standard Legendrian unknot with $\tb(L)=-1$, such
  that $L_0$ and $L$ together bound a thin annulus~$A$ whose framing
  gives the Seifert framing on both knots, i.e. the contact framings
  of the knots are one less than the framings provided by the annulus.
  This property guarantees that after a small perturbation, we can
  assume $A$ to be convex.
  We will show that $L \connsum \d \D_{ot}$ is a destabilization of
  $L_0$.
  To begin, notice that the dividing set on $A$ consists of two
  parallel arcs, each connecting $L_0$ and $L$.
  To establish the theorem, we will show that the annulus~$A \connsum
  \D_{ot}$ bounded by $L_0$ and $L \connsum \d \D_{ot}$ can be
  perturbed into a convex surface with dividing set given by one
  boundary parallel arc with endpoints on $L_0$ (see Figure~\ref{fig:
    destab-ot}).
  This will be done in two steps.

  \begin{figure}[ht]
    \labellist
    \small\hair 2pt
    \pinlabel $L_0$ at 47 39
    \pinlabel $L$ at 22 111
    \pinlabel $\d \D_{ot}$ at 179 90
    \pinlabel $\Gamma$ at 91 87
    \endlabellist
    \centering
    \includegraphics[scale=1.0, keepaspectratio]{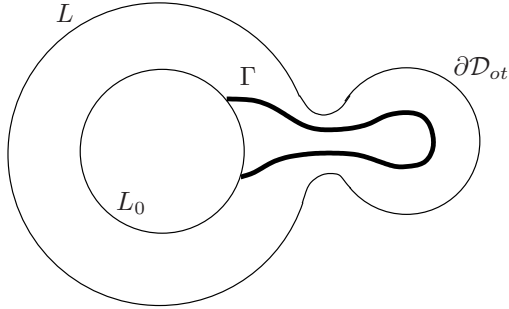}
    \caption{Connected sum with the boundary of an overtwisted disk
      destabilizes a Legendrian knot.}
    \label{fig: destab-ot}
  \end{figure}

  First, assume that the part of $\d \D_{ot}$ where the connected sum
  $L \connsum \d \D_{ot}$ was formed lies in $\d D_l$, and consider
  the Legendrian connected sum $L \connsum \d D_{l}$.
  Make the annulus $A \connsum D_{l}$ convex by a small isotopy fixing
  the boundary (we will keep the same notation for the perturbed
  surface).
  This is possible because $\tb$ of each boundary component is
  negative: $\tb (L_0) = -1$ and $\tb(L\connsum \d D_{l})=\tb(L) +
  \tb(\d D_{l})-1 =-1$. Here, we are using \cite[Proposition 3.1]{Ho}, where the convex perturbation 
  is described as a two-stage process. First, a $C^0$-small isotopy is done near the Legendrian boundary of  
the surface to put a collar neighborhood of the boundary into standard form. Second, a $C^\infty$-small
 isotopy, supported away from the boundary, makes the surface convex. The first perturbation uses a model 
 neighborhood of the Legendrian boundary, and introduces a certain number of singularities with uniform 
 rotation of the contact structure between them. If a collar neighborhood for part of the boundary is already in
 standard form, this part can be kept fixed during the first isotopy. (A close examination of Honda's proof shows that the same arguments go through in our case even though the  
boundary is only piecewise smooth. The key observation here is that the smooth Legendrian knots approximating 
the boundary of $D_l$ forms an annulus which is convex and in standard form in the terminology of \cite[Proposition 3.1]{Ho}.)   
  Therefore, we can assume that the isotopy that makes $A \connsum D_{l}$ convex 
fixes a neighborhood of the diameter of $D_{ot}$ separating $D_l$ and $D_r$, more 
precisely, that all the straight Legendrian arcs between the two half-disks are kept 
fixed, as well as the half-disk $D_r$.  
  The convex surface $A \connsum D_{l}$ (or rather, a surface
  with the same characteristic foliation) can be found in a tight
  contact manifold.
  Indeed, the dividing set on $D_l$ is the same as that on a bypass.
  Bypasses exist in tight contact manifolds, thus by Giroux
  Flexibility we can find a surface with foliation isomorphic to that
  on $D_l$.
  Furthermore, in a tight neighborhood of a bypass we can find both an
  annulus bounded by two small unknots and the strip needed to form
  the connected sum.
  Now, consider the dividing set $\Gamma$ on $A \connsum D_{l}$.
  Tightness of a neighborhood of $A \connsum D_{l}$ implies that
  $\Gamma$ can have no homotopically trivial closed components.
  By \eqref{eq: tb-rot-G}, $\Gamma$ intersects both components of the
  boundary of $A \connsum D_{l}$ at two points.
  This gives the following possibilities for the dividing set: either
  (i) $\Gamma$ consists of two arcs, each connects a point on $L_0$ to
  a point on $L\connsum \d D_{l}$, or (ii) $\Gamma$ consists of two
  boundary-parallel arcs, plus possibly a number of closed curves
  running along the core of the annulus.
  To rule out the second possibility, we argue as follows.
  Since $L_0$ is a small unknot, we can find a disk~$U$ bounded by
  $L_0$, such that $U$ is contained in the tight neighborhood of the
  bypass~$D_l$.
  We can assume that $U$ is convex; then its dividing set is given by
  a single arc connecting two points on $L_0$.
  Moreover, we can assume that $U \cup A \connsum D_{l}$ is a convex
  surface (with a tight neighborhood).
  But then, if there is a boundary-parallel dividing curve on $A
  \connsum D_{l}$ connecting two points of $L_0$, the surface $U \cup
  A \connsum D_{l}$ would have a homotopically trivial closed
  component of the dividing set, which contradicts the tightness of
  the neighbourhood.
  Once we know that on $A \connsum D_{l}$ the dividing set~$\Gamma$
  consists of two parallel arcs running from $L_0$ to $L\connsum \d
  D_{l}$, it is clear that after we attach the other half of the
  overtwisted disk, the convex surface $A \connsum \D_{ot}$ will have
  the dividing set as in Figure~\ref{fig: destab-ot}.
  In conclusion, the knot $L\connsum \D_{ot}$ is a destabilization of
  $L$.
\end{proof}

\section{Higher dimensions: definition of the plastikstufe}
\label{sec:plastikstufe}

Many obstructions of fillability of higher dimensional contact
manifolds have been found in the recent past \cite{Nie, MNW}.  
These obstructions are
all modeled on the overtwisted disk and  take the shape of particular submanifolds in the
given contact manifold.
The initial incarnation of this type of obstruction is the
plastikstufe \cite{Nie}:

\begin{definition}\label{def: plastikstufe}
  Let $(M,\xi)$ be a contact manifold of dimension~$2n+1$, and let $B$
  be a closed $(n-1)$-manifold.
  A \defin{plastikstufe~$\PP_B$ with core~$B$} is a submanifold
  \begin{equation*}
    \PP _B=   D^2 \times B \hookrightarrow M
  \end{equation*}
  such that $\xi \cap T\PP_B$ is a singular foliation that is tangent
  to the fibers $\{z\}\times B$ for every $z \in D^2$, and that
  restricts on every slice $D^2\times \{b\}$ to the foliation of the
  overtwisted disk sketched in Figure~\ref{fig: ot-fol}.
  A plastikstufe~$\PP_B$ is called a \defin{small plastikstufe} if
  there is an embedded open ball in $(M,\xi)$ containing $\PP_B$.
  A contact manifold that contains a plastikstufe is called
  \defin{$PS$-overtwisted}.
\end{definition}

\begin{figure}[ht]
  \includegraphics[width=5cm, keepaspectratio]{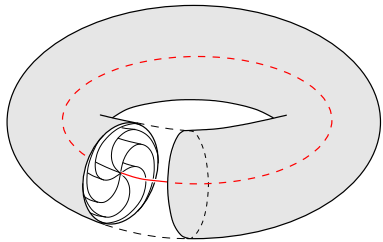}
  \caption{The core of a plastikstufe in a contact $5$-manifold is
    always a circle.}
  \label{fig: plastikstufe-dim5}
\end{figure}

A plastikstufe determines the germ of the contact
structure on its neighborhood \cite{Nie-habil}.
More precisely, let $\D_{ot}$ be an overtwisted disk in
$\bigl(\RR_{ot}^3, \alpha_{ot}\bigr)$ with $\alpha_{ot} = \cos r\, dz
+ r\sin r \, d\varphi$ written in cylindrical coordinates.
Then a plastikstufe~$\PP_B$ has a neighborhood~$U_{PS}$
contactomorphic to a neighborhood of $\D_{ot} \times
\{\text{0-section}\}$ in
\begin{equation*}
  \bigl(\RR_{ot}^3 \times T^*B,
  \alpha_\PS= \alpha_{ot} + \lcan\bigr)
\end{equation*}
with $\lcan = - {\boldsymbol p}\, d {\boldsymbol q}$ the canonical
$1$-form on $T^*B$.
(For purposes of the present paper, we could even accept the existence of
a standard neighborhood of the plastikstufe as part of definition of a
$PS$-overtwisted manifold.)
In the next section, we will prove that Legendrians become flexible in
the presence of a plastikstufe, at least if the latter satisfies some
technical conditions.

\section{$h$-principles, Legendrians, and plastikstufe}
\label{sec:hprinciples}
In this section, we review the definition of loose Legendrians and
prove Theorem~\ref{theorem: p-loose}.
We begin with some background on Gromov's $h$-principle for Legendrian
immersions \cite{PDR, EM}.

\subsection{Formal homotopy of Legendrian immersions and Gromov's
  $h$-principle}

Observe that a Legendrian immersion $f\colon S \to (M, \xi)$ induces a
monomorphism $df\colon TS \to TM$ such that $d_x f(T_x S)$ is a
Lagrangian plane in the symplectic space $\bigl(\xi_{f(x)}, d
\alpha\bigr)$ for every $x\in S$.
If we choose a compatible almost complex structure $J$ on $(M, \xi)$,
then $d_x f (T_x S)$ is a totally real subspace in $\bigl(\xi_{f(x)},
J\bigr)$.
Thus, we can think of $df\colon TS \to \xi$ as Lagrangian
monomorphism, or as a totally real monomorphism if an almost complex
structure is chosen.
(Note that in the latter case, the complexification $df^{\CC}\colon T
_{\CC} S \to \xi_{f(x)}$ with $T_{\CC}S = TS\otimes \CC$ gives a
complex isomorphism in every fiber.)
Suppose that $g,h\colon S \to M$ are two Legendrian immersions.
If there is a regular homotopy of Legendrian immersions~$f_t$ that
connects $g$ and $h$, then the maps $dg \colon T S \to \xi$ and $dh
\colon T S \to \xi$ are homotopic through $F_t = df_t$.
Conversely, assume that
\begin{itemize}
\item [(a)] there is a homotopy of continuous maps $f_t\colon S \to
  M$ connecting $g$ and $h$,
\item [(b)] and a homotopy $F_t$ of Lagrangian monomorphisms over
  $f_t$ connecting $dg, dh\colon T S \to \xi$.
\end{itemize}
In this case, we say that the Legendrian immersions $g$ and $h$ are
\defin{formally homotopic}.
Gromov's $h$-principle for Legendrian immersions implies
that $g$ and $h$ are homotopic through Legendrian immersions whenever
they are formally homotopic.
%


%
We also recall the definition of formal isotopy of isotropic
\emph{embeddings}.
We say that isotropic embeddings $g, h\colon S \to M$ are
\defin{formally isotopic} if
\begin{itemize}
\item [(a')] there is a smooth isotopy  $f_t\colon S \to
  M$ of embeddings connecting $g$ and $h$;
\item [(b')] there is a homotopy $F_t$ of isotropic monomorphisms
  over $f_t$ connecting $dg, dh\colon T S \to \xi$;
\item [(c')] the path of isotropic monomorphisms $F_t$ is homotopic
  to $df_t$ through paths of monomorphisms with fixed endpoints $dg$,
  $dh$.
\end{itemize}
By a result of Gromov \cite{PDR}, formally isotopic {\em subcritical} isotropic embeddings are isotropically isotopic. The same  holds for 
{\em open} Legendrian embeddings.   In general,  this $h$-principle fails for {\em closed} Legendrian embeddings, but
it holds for the special
class of \emph{loose} Legendrians \cite{Mur}, which we review in the next
subsection.
%
%
%

\begin{remark}\label{rmk: almost_complex_structures formal_homotopy}
  The notion of formal homotopy carries over verbatim to the case
  where $\xi$ is an almost contact structure.
  Moreover, consider a path of almost contact structures~$\xi_t$
  that starts and ends at honest contact structures~$\xi_0$ and
  $\xi_1$, respectively.
  Suppose that $f, g\colon S \to M$ are embeddings that are Legendrian
  for both $\xi_0$ and $\xi_1$, and the bundle maps~$dg$ and $dh$ are
  Lagrangian for every almost contact structure~$\xi_t$.
  Then $f$ and $g$ are formally homotopic with respect to $\xi_0$ if
  and only if they are formally homotopic with respect to $\xi_1$.
  In this case, Gromov's $h$-principle implies that $f$ and $g$ are
  homotopic through immersions that are Legendrian with respect to
  $\xi_0$ if and only if they are homotopic through Legendrian
  immersions with respect to $\xi_1$.
\end{remark}

\begin{remark}\label{rmk: rot-class}
  It is useful to reinterpret condition~(b) in the definition of
  formal homotopy in different terms for the case where the image of
  Legendrian immersions $g, h\colon S \to M$ is contained in some open
  ball $U \subset M$.
  Fix a complex structure~$J$ on $\xi$ that is tamed by the conformal
  symplectic structure given by $\restricted{d\alpha}{\xi}$, and
  choose  a $J$-complex trivialization of $(\xi, J)$ over the ball~$U$;
  both $J$ and the trivialization will be unique up to homotopy.
  In particular, these choices allow us to identify
  $\bigl(\restricted{\xi}{U}, J\bigr)$ with the trivial
  bundle~$\CC^n\times U\to U$.
  If $g\colon S \to (U, \xi)$ is a Legendrian immersion in $U$, we can
  write the differential~$dg$ with respect to the trivialization
  chosen above, as a map~$dg\colon TS \to \CC^n$, and furthermore
  since $dg (T_x S)$ is a totally real subspace in $\bigl(\xi_{g(x)},
  J\bigr)$, the complexification
  \begin{equation*}
    dg^\CC\colon T_{\CC}S \to \CC^n
  \end{equation*}
  gives a complex isomorphism in every fiber.
  Given now a second Legendrian immersion $h\colon S \to U$, we can
  relate the two maps $dg^\CC$ and $dh^\CC$ by a map $\varphi\colon S
  \to \GL(n,\CC)$ such that
  \begin{equation}\label{eq: phi}
    dh^\CC_x = \varphi(x)\cdot dg^\CC_x
  \end{equation}
  for every $x\in S$.
  The homotopy class of $\varphi$ is called the \defin{relative
    rotation class} of the Legendrian immersion $h$ relative to $g$.
  The relative rotation class is an element of $[S, \GL(n, \CC)] \cong
  [S, \U(n)]$.
  It is clear that the maps $dg^{\CC}\colon T S \otimes \CC \to \xi$
  and $dh^{\CC}\colon T S \otimes \CC \to \xi$ are homotopic through
  fiberwise complex isomorphisms if and only if their relative
  rotation class vanishes.
\end{remark}


\subsection{Loose Legendrians}

We define loose Legendrian embeddings \cite{Mur} (or loose Legendrians
for short) by requiring that they possess the following model chart.

\begin{definition}\label{def: loose}
  Let $n > 1$.
  In $\bigl(\RR^3, \xi_\Std = \ker(dz-y\, dx)\bigr)$, let $L_0$
  be the Legendrian curve
  \begin{equation*}
    t \mapsto \Bigl(t^2, \frac{15}{4}\, (t^3 - t),
    \frac{3}{2}t^5 - \frac{5}{2}t^3\Bigr) \;,
  \end{equation*}
  where $t$ is in an open interval containing $[-\frac{\sqrt 5}{\sqrt
    3}, \frac{\sqrt 5}{\sqrt 3}]$.
  Let $U \subseteq \RR^3_\Std$ be some convex open set containing
  $L _0$ (see Figure~\ref{fig: stab-strand}).
  In $T^*\RR^{n-1}$ let $Z$ be the Lagrangian zero section
  $\{\mathbf{p} = 0\}$, and let $V_\rho$ be the open set $\bigl\{
  \abs{\mathbf{p}} < \rho, \abs{\mathbf{q}} < \rho \bigr\}$.
  Then $U \times V_\rho$ is canonically identified with an open set in
  $(\RR^{2n+1}, \xi_\Std)$, and $L_0 \times Z$ is  Legendrian
  in this contact Darboux chart.
  If $\rho > 1$, we call the relative pair $(U \times V_\rho, L_0
  \times Z)$ a \defin{loose chart}.
A connected Legendrian manifold $\Lambda \subseteq (M, \xi)$ with
$\dim \Lambda > 1$ is called \defin{loose} if there is a Darboux chart
$U \subseteq M$ so that $(U, \Lambda \cap U)$ is a loose chart.
%
%
\end{definition}

Note that the above definition contains a size restriction $\rho>1$ on
the chart parameter~$\rho$.
This is a key condition: indeed, any Legendrian can be shown to have a
model chart with arbitrarily small~$\rho$.
It is not a priori clear why a chart with a small $\rho<1$ is not
isomorphic to a chart with large $\rho>1$, but this must be true
because loose Legendrians satisfy the following $h$-principle:

\begin{theorem}[\cite{Mur}]
  If two loose Legendrians~$\Lambda _0,  \Lambda _1$ in a contact manifold of
  dimension $2n+1>3$ are formally isotopic, then they are isotopic
  through Legendrian embeddings. \qed
\end{theorem}

In other words, loose Legendrians are flexible, i.e. they are
classified up to Legendrian isotopy by easy to calculate invariants
coming from smooth topology and bundle theory \cite{Mur}.
%
Recall that for general Legendrian embeddings, a similar $h$-principle
manifestly does not hold in any dimension.
It is known that holomorphic curve invariants detect Legendrian
rigidity in many examples where two knots are formally isotopic but
Legendrian non-isotopic.
(This also tells us that the requirement $\rho>1$ imposes a
non-trivial restriction.)
%

\subsection{The proof of Theorem~\ref{theorem: p-loose}}
Our strategy for proving Theorem~\ref{theorem: p-loose} is quite simple.
The motivation comes from Theorem~\ref{theorem: ot-disk destabilizes}:
in every overtwisted $\RR^3$-slice we can isotop a given Legendrian
embedding of a knot to an embedding where its front projection has a
``kink''.
We can think of a plastikstufe as a ``product'' family of overtwisted
disks and perform a family of the above isotopies for slices $\Lambda \cap
\RR^3\times \{b\}$ of a given Legendrian~$\Lambda $ (near the plastikstufe).
This product isotopy would produce a chart required by
Definition~\ref{def: loose} and verify that the given Legendrian is,
indeed, loose.
For us to carry out this plan, the Legendrian~$\Lambda $, perhaps after an
isotopy, must contain a codimension~$0$ submanifold~$\Lambda _0$
diffeomorphic to $B \times[0,1]$, where $B$ is the core of the
plastikstufe~$\PP_B$.
Moreover, $\Lambda _0$ must be in a ``product position'' in some standard
neighborhood $U_{PS}= N_{\epsilon} \D_{ot} \times T_{\epsilon}^*B $
of $\PP_B$:
\begin{eqnarray*}
  \PP_B = & \D_{ot} \times \{ \text{$0$-section} \} &\subset U_{PS} \;, \\
  \Lambda _0 = & K \times \{ \text{$0$-section} \} &\subset U_{PS} \;,
\end{eqnarray*}
where $K \subset N_{\epsilon} \D_{ot}$ is a Legendrian arc.
In other words, we must be able to move the given Legendrian $\Lambda $ by an
ambient contact isotopy so that a product strip of the form $K \times
\{\text{$0$-section} \}$ would be contained in $\Lambda $.
In Lemma~\ref{lemma: isotopy}, we will show that this is possible if
the plastikstufe has a spherical core and {\em trivial rotation}, a
notion we now define.
%

%
A \defin{leaf ribbon} in a plastikstufe~$\PP_B$ with core~$B$ is
  a thin ribbon diffeomorphic to $B \times (0,1)$, obtained by
  shrinking a leaf of $\PP_B$ within itself into the standard
  neighborhood of the core~$B$.
Notice that all such ribbons are Legendrian isotopic: ribbons
contained in the same leaf are all deformations of the same
  leaf and ribbons in different leaves can be related by rotating around
the binding.
From now on, we restrict to plastikstufes with core~$B$ diffeomorphic
to $S^{n-1}$ (see Remark~\ref{rmk: general-B} for a brief discussion
of the general case).
We will compare the formal Legendrian homotopy class of  their
leaf ribbons to the class represented by a punctured Legendrian disk.

\begin{definition}
  For a small plastikstufe~$\PP$ with spherical core, we define the
  rotation class of $\PP$ to be the relative rotation class between a
  leaf ribbon of $\PP$ and a punctured Legendrian disk.
  We say that $\PP$ has \defin{trivial rotation} if this class
  vanishes.
\end{definition}

It is not hard to show that any two Legendrian disks in $(M, \xi)$ are
Legendrian isotopic.
Therefore, the rotation class of a small plastikstufes with spherical core
is well-defined. 

\begin{remark}\label{rmk:pi=0}
  The rotation class of a plastikstufe with core $B=S^{n-1}$ is an
  element of
  \begin{equation*}
    \bigl[B\times [0,1], \GL(n, \CC)\bigr] \cong [B, \U(n)]=
    \pi_{n-1} (\U(n))= \begin{cases} \ZZ & \text{if $n$ is even} \\ 0
      & \text{if $n$ is odd.} \end{cases}
  \end{equation*}
  Thus, in contact manifolds of dimension $2n+1=4m+3$, every small
  plastikstufe with spherical core has trivial rotation.
  We will see in Subsection~\ref{subsection: triv-rot} that in the appropriate
  sense, the rotation of a plastikstufe in dimensions $2n+1=4m+1>5$ is
  determined by the rotation of its core, while in dimension~$5$ it is
  the rotation of the leaf direction that determines the rotation of
  $\PP$.
\end{remark}

Given a plastikstufe $\PP$ with trivial rotation, and a Legendrian~$\Lambda $,
we now isotop $\Lambda$ towards $\PP$.

\begin{lemma}\label{lemma: isotopy}
  Suppose $(M^{2n+1}, \xi)$ is $PS$-overtwisted and $\Lambda  \subset M$ is a
  given Legendrian.
  Assume that there is a small plastikstufe~$\PP_B \subset (M, \xi)$ with
  spherical core $B=S^{n-1}$, and $\PP$ has trivial rotation.
 Then there exists an ambient contact isotopy of $M$ that takes a
 submanifold $\Lambda_0$  of $\Lambda $ diffeomorphic to $S^{n-1}\times [0,1]$ to a
 product strip $\Lambda _1 = K \times \{ \text{$0$-section} \}$ near $\PP_B$.
\end{lemma}

\begin{proof}
  We can assume that  the plastikstufe~$\PP$ and at least some
  part of $\Lambda $ are contained in an open ball~$U$.
  Fix a strip $\Lambda_1 = S^{n-1} \times I$ which is isotopic to a leaf
  ribbon for $\PP$ but lies in the complement of $\PP$. (Here and
  below, $I$ denotes a closed interval.)
  Choose a small Legendrian disk in $\Lambda \cap U$, and let $\Lambda_0 =
  S^{n-1}\times I$ be an annular Legendrian strip in this disk, so that the interior 
  of $\Lambda_0$ is isotopic to the punctured disk.

Now we have two Legendrian embeddings $f_0, f_1\colon S^{n-1}\times I \to M$, such that $\Lambda_i$ 
is the image of $f_i$, $i=0,1$. We would like to connect these by a path of Legendrian 
embeddings. We use Gromov's $h$-principle for subcritical isotropic and {\em open} Legendrian embeddings \cite{PDR}, which in our situation works as follows.

We restrict attention to the core spheres $S_i= f_i \left(S^{n-1} \times \{ c \}\right)$, $i=0,1$, where $c$ is a point inside $I$. (In what follows, we will be assuming that $c$ is close to an endpoint of 
$I$. It will be clear from the context where we want the core spheres to be; a particular choice is unimportant, since any two are isotropically isotopic.)
Let  $\restricted{f_i}{\text{core}}\colon S^{n-1} \to M$ denote the isotropic 
embeddings given by the  corresponding 
restriction maps. We first construct a formal isotropic isotopy 
between these embeddings.

Reparameterizing $I$, we can think of $f_0$ and $f_1$ as maps
$$
f_0\colon S^{n-1} \times [0, \delta] \to M, \qquad 
f_1\colon S^{n-1} \times [1-\delta, 1] \to M
$$ 
for some small $\delta>0$.
Choosing a trivialization of $\xi$ in $U$
   as in Remark~\ref{rmk: rot-class},
   we consider the maps 
   $$
   df_0^\CC \colon T_{\CC}(S^{n-1}\times [0, \delta]) \to \CC ^n, \qquad    df_1^\CC \colon T_{\CC}(S^{n-1}\times [1-\delta, 1]) \to \CC ^n. 
$$   
   The map $df_0^\CC$ is homotopic to the map 
  $G^\CC_0\colon T(S^{n-1}\times [0, \delta]) \to \CC ^n$ defined by
  $G^\CC_0(x, t)= df_0^\CC(x, \delta)$, $0 \leq t \leq \delta$. Similarly, $df_1^\CC$ is homotopic to the map 
  $G^\CC_1\colon T(S^{n-1}\times [1-\delta, 1]) \to \CC ^n$ defined by
  $G^\CC_1(x, t)= df_1^\CC(x, 1-\delta)$, $1- \delta \leq t \leq 1$. Then, we can write 
  \begin{equation} \label{eq: psi}
    df_0^\CC(x, \delta)= \psi(x)   df_1^\CC(x, 1-\delta)
  \end{equation}
  for a map $\psi\colon S^{n-1} \to \GL(n,\CC)$, cf. Remark~\ref{rmk:
    rot-class}.
  Because $\PP$ has trivial rotation and $\Lambda _1$ is isotopic to a leaf
  ribbon of $\PP$, it follows that $\Lambda_0$ and $\Lambda _1$ are formally
  Legendrian homotopic.
  Then, the map $\psi$ is homotopic to the map sending each point of
  $S^{n-1}$ to the identity transformation in $\GL(n,\CC)$. Using 
  this homotopy, we can construct a map  
\begin{equation*}
    G^\CC \colon T_{\CC} (S^{n-1}\times [0,1]) \to \CC ^n 
  \end{equation*}
which gives a complex isomorphism in every fiber and coincides with $df_0^\CC$ near $S^{n-1} \times \{0\}$ and with 
$df_1^\CC$ near $S^{n-1} \times \{1\}$. The real part $G$ of $G^\CC$ is a Lagrangian
monomorphism $T (S^{n-1}\times [0,1]) \to \xi$, restricting to $df_0$
and $df_1$ near the ends of the cylinder $S^{n-1}\times [0,1]$.

 The Smale-Hirsch immersion theorem
(see e.g.  \cite{EM}) implies that $G$ is homotopic to the differential 
$dg$ of an immersion $g\colon {S^{n-1}\times [0,1]} \to M$ such that $g=f_0$ near $S^{n-1}\times \{0\}$
and $g=f_1$ near $S^{n-1}\times \{1\}$.

But $S^{n-1}\times [0,1]$ has dimension $n$, and $M$ has dimension $2n+1$, so we can perturb 
$g$ into general position by homotoping it (through immersions)
to an embedding $\tilde{g}$. We can assume that this homotopy fixes the cylinder near the ends, and that the core spheres we took are close enough to the ends of the cylinder.  

Clearly, the maps $\tilde{g}_t= \tilde{g}(\cdot, t)$ are smooth embeddings 
connecting $\restricted{f_0}{\text{core}}$ and $\restricted{f_1}{\text{core}}$. The maps 
$G_t = \restricted{G(\cdot, t)}{T(S^{n-1}\times \{t\})}$ are isotropic monomorphisms 
covering $\tilde{g}_t$; moreover,  $G_t$ coincides with 
$\restricted{df_0}{TS^{n-1}}$ resp. $\restricted{df_1}{TS^{n-1}}$  near 0 resp. 1, 
and the path $G_t$ is homotopic to $d\tilde{g}$ rel endpoints. 
In short, we have  a formal isotopy of subcritical isotropic
embeddings. By Gromov's $h$-principle, it follows that 
$\restricted{f_0}{\text{core}}$ and $\restricted{f_1}{\text{core}}$ can be connected 
by a path $\tilde{f}_t$ of isotropic embeddings.

The next step is to upgrade  $\tilde{f}_t$ to a path of {\em framed} isotropic 
embeddings. For an $(n-1)$-dimensional isotropic submanifold $S$ of $(M^{2n+1}, \xi)$, we consider a framing by a 
non-vanishing vector field $X\colon S \to \xi$, such that  $X(s) \in (T_s S)^\perp \setminus T_s S$ for every 
point $s \in S$.  (Here, $(T_s S)^\perp$ stands for the symplectic orthogonal complement of $T_s S$ in 
$(\xi, d\alpha)$; this subspace contains $T_s S$ and has dimension $n+1$.) In other words, we look at framings such that 
at every point of $S$ the direct sum $TS \oplus \Span (X)$ is a Lagrangian subspace of $\xi$.   
For an isotropic submanifold contained in a ball, 
we can use a trivialization to find two vector fields $v, v'$ such that at every point  $(T_s S)^\perp = T_s S \oplus \Span (v, v')$, and think of the vector field $X$ as living in $\Span(v, v')$.
Since $X$ does not vanish, after normalizing we can think of the framing as a map $X\colon S \to   S^1$. 

Notice that the original embeddings $f_0$, $f_1$
of $S^{n-1}\times I$ endow the restrictions 
$\restricted{f_i}{\text{core}}$, $i=0,1$  with a framing $X_i$ given by 
$X_i=df_i(\text{the $I$-direction})$.  We would like to find framings $X_t$ for the isotropic embeddings $\tilde{f}_t$, so that 
the path $(\tilde{f}_t, X_t)$ connects $(\restricted{f_0}{\text{core}}, X_0)$ and $(\restricted{f_1}{\text{core}}, X_1)$. 
If $n>2$, then the maps $X_0\colon f_0(S^{n-1}) \to S^1$ and $X_1\colon f_1(S^{n-1}) \to S^1$ are necessarily homotopic, 
so we can find the desired path $X_t$. 
In the case $n=2$, recall that the differentials $df_0$ and $df_1$ of the embeddings of $S^{n-1}\times I$ were connected by a path of Lagrangian monomorphisms $F_t$. 
Then, $F_t(\text{the $I$-direction})$ is a homotopy between $X_0$ and $X_1$.  

The framed isotropic isotopy of the core spheres $f_0(S^{n-1})$ and $f_1(S^{n-1})$ immediately gives a Legendrian isotopy of the original annuli $\Lambda_0$ and $\Lambda_1$. 
Indeed, $\Lambda_0$ and $\Lambda_1$ can be shrinked and flattened (by a Legendrian isotopy) to thin annuli around their core spheres,  linear in the direction of $X_0$ resp. $X_1$. 
These can be connected by a family of thin Legendrian annuli built on the spheres $\tilde{f}_t(S^{n-1})$, spanning a linear strip in the $X_t$ direction.  By the ambient isotopy theorem, 
(see e.g. \cite[Theorem~2.6.2]{Ge}), the resulting isotopy of Legendrian annuli can be realized by 
 a compactly supported contact isotopy of $(M, \xi)$, concluding the proof.
\end{proof}


\begin{lemma}
  Suppose that $\Lambda $ is a Legendrian in $(M, \xi)$ which lies in the
  complement of a small plastikstufe with spherical binding and
    trivial rotation.
  Then there is an open subset $V$ of $M$ with topology $S^{n-1}
  \times D^{n +2}$ such that the pair $(V, V \cap \Lambda )$ is
  contactomorphic to $(D^3_\Std \times D^*S^{n-1}, \Lambda _0)$, where
  $\Lambda _0$ is the cartesian product of a stabilization in $D^3_\Std$
  and the zero section in $D^*S^{n-1}$.
\end{lemma}

\begin{proof}
  By Lemma~\ref{lemma: isotopy}, after an  appropriate Legendrian isotopy, in the standard neighborhood of the plastikstufe $\RR^3_{OT}
  \times D^*S^{n-1}$ the Legendrian  $\Lambda $ is given by $K \times S^{n-1}$.
  Our entire picture is $S^{n-1}$--equivariant, so we then apply
  Theorem~\ref{theorem: ot-disk destabilizes} simultaneously for each
  point in $S^{n-1}$.
\end{proof}

\begin{lemma}
  Suppose that $\Lambda  \subset (M, \xi)$ is Legendrian, and there exists an
  open set $ V\subset M$ such that $(V, V \cap \Lambda )$ is
  contactomorphic to $(D^3_\Std \times D^*S^{n-1}, \Lambda _0)$, where
  $\Lambda _0$ is the cartesian product of a stabilization in $D^3_\Std$
  and the zero section in $D^*S^{n-1}$.
  Then $\Lambda $ is loose.
\end{lemma}

\begin{proof}
  For any metric on $S^{n-1}$, $D^*S^{n-1}$ contains a subset
  symplectomorphic to $V_\rho = \bigl\{\abs{\q} < \rho,\, \abs{\p} <
  \rho \bigr\} \subset T^*\RR^{n-1}$, for some $\rho$.
  Fix $\epsilon > 0$ so that $\epsilon < \rho$.
  In $D^3_\Std$, we can isotope any stabilization by compactly
  supported isotopy to the stabilization $\Lambda_\epsilon$ given by
  \begin{equation*}
    t \mapsto \left(\epsilon t^2, \frac{15 \epsilon}{4}(t^3 - t),
      \frac{\epsilon^2}{2}(3t^5 - 5t^3)\right) \;.
  \end{equation*}
  This defines  a subset~$V$ of $(M , \xi)$ so that $(V, V \cap
  \Lambda )$ is contactomorphic to $(D^3_\Std \times Q_\rho, \Lambda_\epsilon
  \times Z)$, where $Z$ is the zero section in $Q_\rho$.
  Reparametrizing our coordinates by the contactomorphism $(x_i, y_i,
  z) \mapsto (\frac{x_i} {\epsilon}, \frac{y_i}{\epsilon},
  \frac{z}{\epsilon^2})$ demonstrates that $V$ is a loose chart.
\end{proof}

\subsection{Construction of small plastikstufes with trivial
  rotation.} \label{subsection: triv-rot}

For Theorem~\ref{theorem: p-loose} to be useful, we need a supply of
contact manifolds containing small plastikstufes satisfying the
hypotheses of Lemma~\ref{lemma: isotopy}, i.e. having
spherical core and trivial rotation.
We will use the following result of Etnyre--Pancholi \cite{EP} to find
suitable examples.

\begin{theorem}[\cite{EP}]\label{theorem: Etnyre-Pancholi}
  Let $(M, \xi)$ be a contact manifold of dimension $2n + 1$ and let
  $B$ be an $(n - 1)$-dimensional isotropic submanifold with trivial
  conformal symplectic normal bundle.
  Then we may alter $\xi$ to a contact structure $\xi'$ that contains
  a small plastikstufe with core~$B$ such that $\xi$ and $\xi'$ are
  homotopic through almost contact structures.
  In fact, choosing any neighborhood~$U$ of $B$, we can find a smaller
  neighborhood~$U'$ of $B$ such that both the modification and the
  homotopy only affect $\xi$ in $U\setminus \overline{U'}$. \qed
\end{theorem}

In our application we need a refinement of this result, 
in which one also determines the rotation of the plastikstufe.

\begin{prop}\label{prop: PSwithZeroTwist}
  The modification of Theorem~\ref{theorem: Etnyre-Pancholi} can be
  applied in any contact manifold to produce a small plastikstufe with
  spherical core and trivial rotation.
\end{prop}

\begin{proof}
  The only part of the statement which needs justification is that we can
  achieve triviality of the rotation class.
  Indeed, choose a Legendrian disk~$D$ in $(M^{2n+1},\xi)$ and apply
  the construction from the proof of Theorem~\ref{theorem:
    Etnyre-Pancholi} to any cooriented codimension~$1$ submanifold~$B$
  of $D$.
  This $B$ is automatically isotropic, and its symplectic normal
  bundle is spanned by $(X_n, JX_n)$ where $X_n$ denotes a vector
  field in $D$ that is transverse to $B$, and $J$ is a compatible
  almost complex structure on $\xi$.
  Since we can do this construction inside a small ball around the
  Legendrian disk~$D$, the resulting contact structure~$\xi'$ will
  contain a small plastikstufe~$\PP_B$.
  Assume for the rest of the proof that $B$ is a small
  sphere~$S^{n-1}$ that bounds a Legendrian disk in $L$.
  (See Remark~\ref{rmk: general-B} below for discussion of the more
  general case.)
  By Remark~\ref{rmk:pi=0}, the rotation class of a small plastikstufe
  with spherical core necessarily vanishes in contact manifolds of
  dimension $4m+3$, but we still need to study the remaining
  dimensions~$4m+1$.
  Below, we will first show that the rotation for any plastikstufe
  constructed by Theorem~\ref{theorem: Etnyre-Pancholi} 
  (on a core $B=S^{n-1}$ that bounds a Legendrian
  disk before creating the plastikstufe) is trivial if $\dim M  >
  5$. (This phenomenon already manifested itself in the proof of Theorem \ref{theorem: p-loose}, where  any two framings on isotropic spheres were 
  automatically homotopic if $\dim M>5$.)
  Finally, we will show that using the Etnyre--Pancholi
  theorem~\ref{theorem: Etnyre-Pancholi} in dimension~$5$ with more
  care, we can produce plastikstufes with arbitrary rotation class,
  including trivial rotation.
  We can assume that the modification to $\xi'$ has been done in a
  sufficiently small neighborhood of $S^{n-1}$ so that $\xi' = \xi$
  both close to the center of the auxiliary Legendrian disk~$D$
  chosen above, and outside of an embedded ball.
  Moreover, the Etnyre--Pancholi modification does not change the
  contact structure near the chosen isotropic submanifold~$B$.
  Thus, close to $B$ the leaves of the resulting plastikstufe are also
  Legendrian with respect to the \emph{old} contact structure~$\xi$.
  Since $\xi$ and $\xi'$ are homotopic, 
  by  Remark~\ref{rmk: almost_complex_structures formal_homotopy} 
  it suffices to
  examine the rotation of a leaf with respect to $\xi$.
%

%
  Choose a leaf ribbon near $B$ and compare its rotation to the
  punctured Legendrian disk $D - \{0\}\cong B \times I$.
  Recall that by Remark~\ref{rmk: rot-class} we can trivialize $\xi$
  over a ball containing the Legendrians, and then obtain the rotation
  class as the homotopy class of the map~$\varphi\colon B \times I \to
  \GL(n,\CC)$ that satisfies Equation~\eqref{eq: phi}, where $h$ and
  $g$ denote the two embeddings we wish to compare.
  In fact, because $B \times I$ deformation retracts to $B$, it
  suffices to study the restriction of this map to $B\times \{1/2\}$.
  Then the two embeddings are the same over $B$ and only differ in the
  remaining direction.
  Over $B$ the tangent bundle to the leaf ribbon splits as $TB \oplus
  \langle X_l \rangle$, where $X_l$ is the leaf direction, and the
  tangent bundle to the flat annulus splits as $TB \oplus \langle X_n
  \rangle$, where $X_n$ is the normal to $B$ in the disk~$D$.
%
  Therefore, for any $q\in B$ we have
  \begin{equation*}
    \varphi(q) = 
    \begin{pmatrix}
        \mbox{\Huge ${\mathrm {Id}}$} &  \begin{array}{c} * \\ \vdots \\ *
        \end{array}\\
        0 \dotsb 0 & \kappa(q) 
      \end{pmatrix}
  \end{equation*}
  where $\kappa\colon B \to \CC^*$ is a continuous function.
  This helps us understand the homotopy class of $\varphi$ in $[B,
  \GL(n, \CC)]$.
  We consider two cases.
  {\em Case 1.} Suppose $2n+1 >5$.
  Then the function $\kappa \colon B\to \CC^*$ is homotopic to $1$,
  since by assumption $B=S^{n-1}$ with $n-1 \geq 2$.
  It follows that $\varphi$ represents the trivial class in $[B,
  \GL(n, \CC)]$, so the plastikstufe~$\PP_B$ has trivial rotation.
  {\em Case 2.} Suppose $2n+1=5$.
  Now, the function $\kappa$, and therefore $\varphi$, may represent a
  non-trivial homotopy class. It turns out that this class depends on
  the choice of coordinates in the Etnyre--Pancholi construction.
  A reparametrization will allow us to adjust the rotation class. To
  see this, we first review the setup of the construction of
  \cite{EP}.
  Given an isotropic manifold~$B$ with trivial symplectic conormal
  bundle, we use the isotropic neighborhood theorem to identify a
  neighborhood of $B$ in $(M, \xi)$ with $T^*B \times \RR^3$, equipped
  with the contact form $p\, dq + dz + r^2\, d\theta$.
  Choosing a transverse curve~$\gamma(t)$ in $\RR^3$, we can further
  identify a neighborhood of $B \times \gamma$ in $T^*B \times \RR^3$
  with
  \begin{equation*}
    \bigl(T^*B \times S^1 \times D^2,
    \ker (p\, dq + dt + r^2 d\theta)\bigr).
  \end{equation*}
  Here, $B$ is the $0$-section of $T^*B$, $p\, d q$ is the canonical
  $1$-form on $T^* B$, $S^1$ is the transverse curve $\gamma(t)$, and
  $D^2$ is an open disk with polar coordinates~$(r, \theta)$.
  Since $\dim M = 5$, we have $B=S^1$, and the coordinate~$q$
  varies in the unit circle.
  The generalized Lutz twist from the Etnyre--Pancholi construction
  produces in each of the slices $T^*B \times \{t_0\} \times D^2$ a
  plastikstufe~$\PP_{t_0}$ whose core is $B \times \{t_0\}$, where $B$
  is the zero section as before.
  Near the core, the leaves of $\PP_{t_0}$ are given by $\{\theta =
  \text{const} \}$.
  We can define for any $a\in \ZZ$ a contactomorphism
  \begin{equation*}
    F_a\colon T^*B \times S^1 \times \RR^2 \to
    T^*B \times S^1 \times \RR^2
  \end{equation*}
  by $F_a(p, q, t, r e^{i\theta}) = \bigl(p-a r^2, q, t, r
  e^{i\,(\theta + a q)}\bigr)$ that preserves $p\, dq + dz + r^2\,
  d\theta$.
  Note that the zero section $\{p=0, r=0\}$ is preserved by $F_a$.
  Instead of applying now the generalized Lutz twist of \cite{EP} to
  the initial model neighborhood of $B$ we have chosen in the
  Legendrian disk~$D$ above, we use first a restriction of the
  map~$F_a$ to modify the coordinates of the model neighborhood.
  This means that for a fixed $t_0$, the construction will produce a
  plastikstufe with the same core for either choice of coordinates.
  The leaf direction near the binding, however, is different: we get
  $\{\theta= \text{const}\}$ for the plastikstufe $\PP$ constructed
  with the old coordinates, and $\theta'= \theta + a q =\text{const}$
  for $\PP'$ constructed in the new coordinates.
  Intuitively, the leaf direction of $\PP'$ rotates $a$ times with
  respect to the leaf direction of $\PP$ as we move around the circle
  $S^1=B$.
  We now compute the relative rotation class between $\PP$ and $\PP'$.
  For this, we compare two trivializations of the bundle $\xi \to
  B=S^1$ corresponding to the (Legendrian) leaves of our
  plastikstufes.
  With a choice of a compatible almost complex structure~$J$ on $\xi$,
  each plastikstufe yields a trivialization of $\restricted{\xi}{B}$
  given by
  \begin{equation*}
    \text{(core $S^1$-direction}, J(\text{core $S^1$-direction}),
    \text{leaf direction}, J(\text{leaf direction})).
  \end{equation*}
  Since the $S^1$-direction for $\PP$ and $\PP'$ is the same and the
  leaf directions differ as described above, the two trivializations
  are related by the map $\varphi\colon S^1 \to \U(2)$,
  \begin{equation*}
    \varphi(q) = 
    \begin{pmatrix}
      1 & 0 \\
      0 &  2\pi i a q 
    \end{pmatrix} .
  \end{equation*}
  Different values of $a$ give maps lying in different homotopy
  classes in $[S^1, \U(2)]= \pi_1 (\U(2))=\ZZ$.
  More precisely, we can realize any given element in $\pi_1 (\U(2))$ by
  an appropriate choice of $a$.
  In particular, we can realize the class given by a flat Legendrian
  annulus (i.e. the one contained in the Legendrian disk~$D$).
  It follows that by an appropriate choice of coordinates,
  we can construct a plastikstufe with trivial rotation class.
\end{proof}

\begin{remark}
  The previous proposition essentially shows that the Etnyre--Pancholi
  construction can be used to produce a plastikstufe with an arbitrary
  given rotation class.
  Indeed, any rotation class of the core can be realized by Gromov's
  $h$-principle for subcritical isotropic embeddings.
  In the leaf direction, rotation is always trivial in dimensions
  higher than $5$; in dimension~$5$, any given rotation can be
  realized by an appropriate choice of coordinates in the
  Etnyre--Pancholi construction.
  We leave the details to the reader as we have no use for
  plastikstufes with non-trivial rotation.
\end{remark}

\begin{remark}\label{rmk: general-B}
  The construction above could be generalize in a relatively easy way
  for a small plastikstufe~$\PP_B$ with core~$B^{n-1}$, as long as
  $B^{n-1}$ can be (smoothly) embedded into $\RR^n$.
  To simplify the presentation, we have chosen to restrict to the
  spherical case, pointing instead to Question~\ref{quest: true for
    BLOB}.
\end{remark}


\section{Flexible Weinstein manifolds}
\label{sec:loose}

Let $(W^{2n+2}, d\lambda)$ be an exact symplectic manifold.
Given an exhausting Morse function $\varphi\colon W \to \RR$, we say
the triple $(W, \lambda, \varphi)$ is a \defin{Weinstein manifold} if
the vector field~$V$ determined by $V\, \lrcorner \, d\lambda =
\lambda$ is gradient-like for the Morse function~$\varphi$.
(If $\varphi$ has finitely many critical points, we say that $W$ has
finite type.)
It follows that $\lambda$ is a contact form on any non-critical level
set~$\varphi^{-1}(c)$.
Similarly, $(W, \lambda, \varphi)$ is a \defin{Weinstein cobordism}
from $(Y_0, \xi_0)$ to $(Y_1, \xi_1)$ if $\d W = Y_1 \cup - Y_0$,
where $Y_0$ and $Y_1$ are level sets for $\varphi$, and $V$ points
inwards along $Y_0$ and outwards along $Y_1$.
If $Y_0 = \emptyset$, $W$ is called a \defin{Weinstein domain}.
Observe that because $\lie{V} \lambda = \lambda$, the flow of the
vector field~$V$ expands $\lambda$ exponentially.
If $p \in W$ is a critical point of $\varphi$, the \defin{descending
  manifold} $C$ of $p$ consists of all the points in $W$ which
converge to $p$ by the flow of $V$.
%
%
Since any vector in $TC$ converges to the zero vector at the point~$p$
under the flow of $V$, but this flow also expands $\lambda$
exponentially, we conclude that $\restricted{\lambda}{C} = 0$.
Thus any descending manifold is an isotropic submanifold of $W$, and
furthermore it intersects every level set in a contact isotropic
submanifold.
In particular, every critical point of $\varphi$ must have index less
than or equal to $n+1$.
A Weinstein cobordism of dimension $2n+2 > 4$ is called
\defin{flexible} if every descending manifold of index $n+1$
intersects the lower level sets along a \emph{loose} Legendrian sphere
(in other words, attaching spheres for all handles of index $n+1$ are
loose Legendrians.)
Here, we will assume for simplicity that all critical values of
$\varphi$ are distinct.
Using the $h$-principle result from \cite{Mur}, it is shown in
~\cite{CE} that flexible Weinstein cobordisms can be completely
classified by their topology.
There are several related theorem proved in \cite[Chapter 14]{CE}; we
state two results (in the form relevant to our case).
\begin{theorem}[\cite{CE}]\label{thm: symphcob}
  Let $(W, \lambda, \varphi)$ be a flexible Weinstein cobordism from
  $(Y_0, \ker \alpha_0)$ to $(Y_1, \ker \alpha_1)$, such that $W$ is
  smoothly a product cobordism.
  Then after a homotopy through Weinstein structures, the cobordism
  $(W, \lambda, \varphi)$ is isomorphic to $(Y_0 \times \RR, e^t
  \alpha_0, t)$.
  In particular, $(Y_0, \ker \alpha_0)$ is contactomorphic to $(Y_1,
  \ker \alpha_1)$. \qed
\end{theorem}

\begin{theorem}[\cite{CE}]\label{thm: two-Wstructures}
  Let $(W, \lambda_0, \varphi_0)$ and $(W, \lambda_1, \varphi_1)$
  be two flexible Weinstein structures on a fixed manifold~$W$.
  Suppose that $\omega_0 = d\lambda_0$ and $\omega_1 = d\lambda_1$ are
  homotopic through non-degenerate $2$-forms.
  Then there exists a homotopy of Weinstein structures connecting
  $(W_0, \lambda_0, \varphi_0)$ and $(W_1, \lambda_1, \varphi_1)$. \qed
\end{theorem}
Now we are ready to prove Theorem~\ref{theorem: exo=std} from
the Introduction:

\begin{proof}[Proof of Theorem~\ref{theorem: exo=std}]
  We will prove the theorem under the assumption that hypothesis~(2)
  is satisfied, i.e.  $(Y_0, \xi_0)$ and $(Y_1, \xi_1)$ are the ends of a
  Stein cobordism~$W$ smoothly equivalent to a product cobordism.
As a first step we consider the induced Weinstein structure on the 
cobordism, and apply Theorem~\ref{thm: symphcob}.
  (The proof that hypothesis~(1) implies the isomorphism of the
  contact structures works in exactly the same way, except that we use
  Theorem~\ref{thm: two-Wstructures} instead of Theorem~\ref{thm:
    symphcob}).
  Let $(M^{2n+1}, \xi_\PS)$ be a $PS$-overtwisted manifold containing
  a plastikstufe $\PP$ satisfying the hypotheses of
  Theorem~\ref{theorem: p-loose}.
  By connect summing this manifold onto each level set of $W$ (away
  from all descending manifolds) we obtain a new Weinstein
  cobordism~$\tilde{W}$ from $(Y_0, \xi_0) \connsum (M, \xi_\PS)$ to
  $(Y_1, \xi_1) \connsum (M, \xi_\PS)$.
  The cobordism $\tilde{W}$ has the same critical points as $W$, and
  each descending manifold is disjoint from $M$ (and in particular,
  from the plastikstufe $\PP$) at every level set.
  By Theorem~\ref{theorem: p-loose} this implies that all descending
  manifolds of index~$n+1$ intersect all level sets in loose
  Legendrians, therefore $\tilde{W}$ is a flexible Weinstein
  cobordism.
  Since $\tilde{W}$ is smoothly equivalent to the product cobordism
  $(Y\connsum M) \times [0,1]$, we can complete the proof by applying
  Theorem~\ref{thm: symphcob}.
\end{proof}

\section{Open questions}

The questions in this section are not new, but with the results of
this paper they may be formulated in a more precise way.
Above we have  showed
 that certain $PS$-overtwisted manifolds are
flexible within the class of $PS$-overtwisted manifolds being Stein
cobordant by a \emph{topologically} trivial cobordism.

\begin{question} Even if the general flexibility question seems  to be out of reach, it would be interesting 
to tackle more special cases. For example: 

 (a)  Is it  true that the connected sum of an Ustilovsky
  sphere (cf. \cite{ustilov})
    with a suitable $PS$-overtwisted manifold is contactomorphic
  to the $PS$-overtwisted manifold? Note that in this case, the two contact manifolds 
  are cobordant via a topologically
  non-trivial Stein cobordism.
  
  (b) Take a model $PS$-overtwisted sphere $(S^{2n+1}, \xi_{PS})$. Is it true that $(S^{2n+1}, \xi_{PS} \connsum \xi_{PS})$ is contactomorphic  to 
   $(S^{2n+1}, \xi_{PS})$? 
   
\end{question}

\begin{question} \label{quest: true for
    BLOB}
     Can we remove the technical conditions in Theorem \ref{theorem: p-loose} and prove similar 
     statements for arbitrary plastikstufes, or, more generally, 
for  arbitary bLobs? (Bordered Legendrian open books, or bLobs for short, 
were introduced in \cite{MNW} to generalize the notion of $PS$-overtwistedness.)

It is not hard to see that the arguments of this paper go through in the presence of a ``large'' open 
ball in $\RR^3_{ot}\times \RR^{2n}$. Moreover, some properties of such ball \cite{NP} resemble 
those 
of a loose chart. Is it true that a neighborhhood of an arbitrary plastikstufe or bLob contains such an open ball? 
\end{question}

A further  question is about finding the correct definition of
overtwistedness in higher dimensions. A very useful approach to studying contact structures in all dimensions is provided by open book decompositions \cite{GiOB}.

\begin{question} Can our results be interepreted in terms of open books to obtain some type of flexibility in that context?
\end{question}


\begin{thebibliography}{AO2}
\bibitem[AS]{AS} M. Abouzaid and P. Seidel, \textit{Altering symplectic
    manifolds by homologous recombination}, arXiv:1007.3281

\bibitem[CE]{CE} K. Cieliebak, Y. Eliashberg, \textit{Symplectic
    Geometry of Stein Manifolds}, book to appear.

\bibitem[Dy]{Dymara} K. Dymara, \textit{Legendrian knots in
    overtwisted contact structures on $S^3$},
Ann. Global Anal. Geom. {\bf19} (2001), 293--305.

\bibitem[E]{EliashbergClassification} Y. Eliashberg,
  \textit{Classification of overtwisted contact structures on
    $3$-manifolds}, Invent. Math. {\bf 98} (1989), no. 3, 623--637.

\bibitem[EF]{EF} Y. Eliashberg and M. Fraser, \textit{Topologically
    trivial Legendrian knots}, J. Symplectic Geom. {\bf 7} (2009),
  no. 2, 77--127.

\bibitem[EM]{EM} Y. Eliashberg and N. Mishachev, \textit{Introduction to
    the {$h$}-principle}, Graduate Studies in Mathematics {\bf 48},
  American Mathematical Society (2002)
  
\bibitem[Et1]{EtnyreKnots} J. Etnyre, \textit{Legendrian and
    Transversal Knots}, in the Handbook of Knot Theory (Elsevier
  B. V., Amsterdam), (2005) 105--185.

\bibitem[Et2]{Et} J. Etnyre, \textit{On knots in overtwisted contact
    structures}, Journal of Quantum Topology, to appear,
  arXiv:1012.3745.

\bibitem[EP]{EP} J.~Etnyre and  D.~Pancholi \textit{On generalizing Lutz
    twists}, J. London Math. Soc. {\bf 84} (2011), no. 3, 670--688.
  
  
\bibitem[Ge]{Ge} H. Geiges, \textit{An introduction to contact
    topology}, Cambridge Studies in Advanced Mathematics {\bf 109},
  Cambridge University Press (2008).

\bibitem[Gi1]{Gi} E. Giroux, \textit{Convexité en topologie de
    contact}, Comment. Math. Helv. {\bf 66} (1991), no. 4, 637--677.

\bibitem[Gi2]{GiOB} E. Giroux, \textit{Géométrie de contact: de la
    dimension trois vers les dimensions supérieures}, Proceedings of
  the {I}nternational {C}ongress of {M}athematicians, {V}ol. {II}
  (2002), 405--414.
  
\bibitem[Gr]{PDR} M. Gromov, \textit{Partial differential relations},
  Ergebnisse der Mathematik und ihrer Grenzgebiete (3), {\bf 9},
  Springer-Verlag (1986)
  
\bibitem[Ho]{Ho} K. Honda, \textit{On the classification of tight
    contact structures. I}, Geom. Topol. {\bf 4} (2000), 309--368.
  
\bibitem[MNW]{MNW} P. Massot, K. Niederkr\"uger and C. Wendl, \textit{Weak
    and strong fillability of higher dimensional contact manifolds},
  Invent. Math. (2012), doi: 10.1007/s00222-012-0412-5
   
\bibitem[Mc]{Mc} M. McLean, \textit{Lefschetz fibrations and
    symplectic homology}, Geom. Topol. {\bf 13} (2009), 1877-–1944
  
\bibitem[Mu]{Mur} E. Murphy, \textit{Loose Legendrian Embeddings in
    High Dimensional Contact Manifolds}, arXiv:1201.2245.
  
\bibitem[Ni1]{Nie} K. Niederkr\"uger, \textit{The plastikstufe -- a
    generalization of the overtwisted disk to higher dimensions},
  Algebr. Geom. Topol.  {\bf 6} (2006), 2473--2508.
  
\bibitem[Ni2]{Nie-habil} K. Niederkr\"uger, Habilitation, in
  preparation.

\bibitem[NP]{NP} K. Niederkr\"uger and F. Presas, \textit{Some remarks on
    the size of tubular neighborhoods in contact topology and
    fillability}, Geom. Topol. {\bf 14} (2010), no. 2, 719--754.
  
  
\bibitem[Us]{ustilov}  
I.  Ustilovsky, 
\textit{Infinitely many contact structures on $S^{4m+1}$}, 
Internat. Math. Res. Notices {\bf1999}, no. 14, 781--791. 
  
  
\bibitem[Vo]{Vog} T. Vogel, \textit{Existence of Engel structures},
  Ann. of Math. (2) {\bf 169} (2009), no. 1, 79--137.
\end{thebibliography}
\end{document}